\documentclass{amsart}
\usepackage{amsmath,amssymb,amsfonts,amsthm,amscd,indentfirst}
\usepackage{latexsym}
\usepackage{hyperref}\usepackage{xcolor}

\newtheorem{thm}{Theorem}[section]
\newtheorem{lem}[thm]{Lemma}
\newtheorem{cor}[thm]{Corollary}
\newtheorem{rem}[thm]{Remark}

\newtheorem{defi}[thm]{Definition}

\newcommand{\ba}{\begin{array}}
\newcommand{\ea}{\end{array}}

\def \qed{\cqfd}

\newcommand*{\QEDB}{\hfill\ensuremath{\square}}
\def\qed{\vbox{\hrule
\hbox{\vrule\hbox to 5pt{\vbox to 8pt{\vfil}\hfil}\vrule}\hrule}}

\newcommand{\beg}{\begin{equation*}}
\newcommand{\begn}{\begin{equation}}
\newcommand{\en}{\end{equation*}}
\newcommand{\enn}{\end{equation}}

\begin{document}
\title[Ricci flow starting from an embedded closed convex surface in $\mathbb{R}^3$]{Ricci flow starting from an embedded closed convex surface in $\mathbb{R}^3$}
\author{Jiuzhou Huang and Jiawei Liu}
\address{Jiuzhou Huang\\ Department of Mathematics and Statistics\\ McGill University\\ Montreal\\ Quebec\\ H3A 0B9\\ Canada.} \email{jiuzhou.huang@mail.mcgill.ca}
\address{Jiawei Liu\\ Department of Mathematics and Statistics\\ McGill University\\ Montreal\\ Quebec\\ H3A 0B9\\ Canada.}
 \email{jiawei.liu@ovgu.de}
\thanks{AMS Mathematics Subject Classification. }
\thanks{}

\maketitle

\begin{abstract}
In this paper, we establish the existence and uniqueness of Ricci flow that admits an embedded closed convex surface in $\mathbb{R}^3$ as metric initial condition. The main point is a family of smooth Ricci flows starting from smooth convex surfaces whose metrics converge uniformly to the metric of the initial surface in intrinsic sense.
\end{abstract}

\section{Introduction}

Since Ricci flow was introduced by Hamilton \cite{Hamilton} in 1982, it has many applications in differential geometry and topology, such as the solutions of Poincar\'e conjecture \cite{Perelman1,Perelman2} and differentiable sphere theorem \cite{SchoenBrendle} etc. Another important application of Ricci flow is the smoothing of initial condition. In \cite{Simon1, Simon2}, Simon proved the existence of Ricci flow that admits a class of irregular metric spaces with dimension two or three as metric initial condition. This is an approximation of the metric space by Ricci flow. Based on Simon's work, Richard \cite{Richard1, Richard2} studied the existence and uniqueness of Ricci flow whose metric initial condition is a closed Alexandrov surface with curvature bounded from below, which gives a canonical smoothing of such surface via Ricci flow. The works of Simon and Richard are related to the Gromov-Hausdorff convergence. In this paper, we consider stronger convergence of Ricci flow for embedded closed convex surface in $\mathbb R^3$ (see Theorem \ref{Existence}).

Before stating our results, we first recall Simon \cite{Simon1, Simon2} and Richard's results \cite{Richard1,Richard2}. One key point in \cite{Simon2} is the following estimates.
\begin{thm}\label{simon}(Theorem $7.1$ in \cite{Simon2})
Let $(M,g_0)$ be a complete smooth three (or two) manifold without boundary such that
\begin{equation}\label{0416}
\begin{split}
&(a)\ Ricci(g_0)\geqslant k;\\
&(b)\ vol(^{g_0}B_1(x))\geqslant v_0>0\text{ for all }x\in M;\\
&(c)\ \sup_M|Riem(g_0)|<\infty.
\end{split}
\end{equation}	
Then there are constants $c_1=c_1(v_0, k)>0$, $c_2=c_2(v_0, k)>0$, $S=S(v_0,k)>0$ and $K=K(v_0, k)$ and a solution $(M,g(t))_{t\in[0,T)}$ to Ricci flow which satisfies $T\geqslant S$, and
\begin{equation}
\begin{split}
&(a_t)\ Ricci (g(t))\geqslant -K^2;\\
&(b_t)\ vol(^{g_t}B_1(x))\geqslant \frac{v_0}{2}>0 \text{ for all } x\in M\ and\ t\in (0,T);
\end{split}
\end{equation}	
\begin{equation}
\begin{split}
&(c_t)\ \sup_M|Riem(g(t))|\leqslant \frac{K^2}{t} \text{ for all } t\in(0,T);\\
&(d_t)\ d(p,q,s)-c_2(\sqrt{t}-\sqrt{s})\leqslant d(p,q,t)\leqslant e^{c_1(t-s)}d(p,q,s),\\
&\ \ \qquad \text{ for all } 0<s\leqslant t<T\ and\ p,\ q \in M.
\end{split}
\end{equation}	
(Note that the estimates are trivial for $t=0$.)
\end{thm}

Based on this result, the following existence and uniqueness of Ricci flow were proved in \cite{Simon2, Richard2}.
\begin{thm}\label{richard} (Theorem $1.9$ in \cite{Simon2} and Theorem $0.5$ in \cite{Richard2})
Let $(X,d)$ be a closed Alexandrov surface with curvature bounded from below by $-K$. Then there is a smooth Ricci flow $(M, g(t))_{t\in (0 , T )}$ admitting $(X,d)$ as metric initial condition in the sense that the Riemannian distances $d_{g(t)}$ uniformly converge as $t$ goes to $0$ to a distance $\tilde{d}$ on $M$ such that $(M, \tilde{d})$ is isometric to $(X, d)$. 

Moreover, if there is another smooth Ricci flow $(N, h(t))_{t\in (0 , T )}$ also admitting $(X,d)$ as metric initial condition in above sense, then there is a diffeomorphism $\varphi:M\rightarrow N$ such that $g(t)=\varphi^*h(t)$.
\end{thm}

Their ideas are as follows. They first construct a sequence of smooth manifolds $(M_i,g_i)$ that converges to $(X,d)$ in Gromov-Hausdorff distance and keeps the uniform properties in $(\ref{0416})$. For every $i$, there is a smooth Ricci flow $(M_i, g_i(t))_{t\in[0,T)}$ starting from $(M_i,g_i)$. Then by using Theorem \ref{simon} and taking the limit as $i\to\infty$, they get the Ricci flow $(M,g(t))_{t\in(0,T)}$ that converge to $(X,d)$ in Gromov-Hausdorff distance.

A natural question is when such a Ricci flow will converge to the initial metric in classical sense, and what kind of uniqueness one can claim. In this paper, when the metric initial condition $(X,d)$ is an embedded closed convex surface in $\mathbb{R}^3$, we prove that the Ricci flow will converge in intrinsic sense to $(X,d)$ as $t\to0$ and that such flows keep the isometries between their metric initial conditions. We would like to remark that the convex surface in this paper is in the sense of Alexandrov (see section 2), unless otherwise specified. 

Our first result is the following existence theorem.
\begin{thm}\label{Existence}
If $(X, d)$ is an embedded closed convex surface in $\mathbb{R}^3$, then there exists a $T>0$ and a smooth Ricci flow $(X, g(t))_{t\in(0,T)}$ such that the distance functions $d_{g(t)}$ induced by $g(t)$ converge uniformly to $d$ as $t\to0$, that is, 
\begin{equation}\label{limits}
\lim_{t\to 0} \max_{p,q\in X}|d_{g(t)}(p,q)-d(p,q)|=0.
\end{equation}
\end{thm}

The difference between Theorem \ref{richard} and Theorem \ref{Existence} is that we remove the isometry between $(M,\tilde{d})$ and $(X,d)$ in Theorem \ref{richard} when the metric initial condition $(X,d)$ is an embedded closed convex surface in $\mathbb{R}^3$.  This is due to the existence of smooth convex surfaces that approximate $(X,d)$ in Hausdorff distance (Lemma \ref{approximation C0}) instead of the Gromov-Hausdorff convergence in \cite{Richard1, Simon1, Simon2}. In fact, removing the isometry is crucial for proving the uniqueness of such Ricci flow and then study the rigidity problem of closed convex surfaces (see our project in Section \ref{further discussions}).

From now on, unless otherwise specified, by saying that $(X,g(t))_{t\in(0,T)}$ is a Ricci flow admitting an embedded closed convex surface $(X,d)$ in $\mathbb{R}^3$ as metric initial condition, we mean that it is a Ricci flow in the sense of Theorem \ref{Existence}.

Since every closed convex surface can be embedded in $\mathbb{R}^3$ as the boundary of a  convex body by using Alexandrov's embedding theorem (Theorem \ref{Alex embedding}). By Theorem \ref{Existence}, we have the following existence result.

\begin{cor} \label{Existence0}
For any closed convex surface $(X, d)$, there exists a $T>0$ and a smooth Ricci flow $g(t)$ with ${t\in(0,T)}$ admitting $(X,d)$ as metric initial condition in the sense that $(\tilde{X},g(t))_{t\in(0,T)}$ is a Ricci flow admitting $(\tilde{X},\tilde{d})$ as metric initial condition, where $(\tilde{X},\tilde{d})$ is an isometric embedding of $(X,d)$ into $\mathbb{R}^3$.
\end{cor}
The second result in this paper is the following uniqueness theorem.
\begin{thm}\label{uni1}
Assume that $(X_1,d_1)$ and $(X_2,d_2)$ are two non-degenerate embedded closed convex surfaces in $\mathbb{R}^3$ and $f: (X_1,d_1)\rightarrow(X_2,d_2)$ is an isometry. Let $(X_1,g_1(t))_{t\in(0,T)}$ and $(X_2,g_2(t))_{t\in(0,T)}$ be Ricci flows admitting $(X_1,d_1)$ and $(X_2,d_2)$ as metric initial conditions respectively. Then $g_1(t)=f^*g_2(t)$. 
\end{thm}

Theorem \ref{uni1} gives the exact expression of the diffeomorphism in Theorem \ref{richard} when the metric initial condition is an embedded closed convex surface in $\mathbb{R}^3$. This result means that Ricci flows obtained in Theorem \ref{Existence} keep the isometries between their metric initial conditions. The point here is to prove that the isometry between the two metric initial conditions is differentiable (Theorem \ref{01010}), which implies that the pull back metrics under this isometry still satisfy Ricci flow. Then Theorem \ref{uni1} follows from Proposition $0.6$ in \cite{Richard2}.

If the metric initial conditions are closed convex surfaces, we have the following uniqueness result.

 \begin{cor}\label{uni}
 Assume that $(X_1,d_1)$ and $(X_2,d_2)$ are two isometric closed convex surfaces with non-degenerate isometric embeddings in $\mathbb{R}^3$, and that $g_1(t)$ and $g_2(t)$ with $t\in(0,T)$ are Ricci flows admitting $(X_1,d_1)$ and $(X_2,d_2)$ as metric initial conditions in the sense of Corollary \ref{Existence0}. Then $g_1(t)$ and $g_2(t)$ are isometric. 
\end{cor}

\begin{rem}
Let $\varphi_1:(X_1,d_1)\rightarrow(\tilde{X}_1, \tilde{d}_1)$ and $\varphi_2:(X_2,d_2)\rightarrow(\tilde{X}_2, \tilde{d}_2)$ be isometric embeddings of $(X_1,d_1)$ and $(X_2,d_2)$ into $\mathbb{R}^3$ respectively, and $f:(X_1,d_1)\rightarrow(X_2,d_2)$ be the isometry in Corollary \ref{uni}. From Theorem \ref{uni1}, $F:=\varphi_2\circ f\circ\varphi_1^{-1}$ is the isometry between $g_1(t)$ and $g_2(t)$ in Corollary \ref{uni}, that is, $g_1(t)=F^*g_2(t)$.
\end{rem}

The motivation of this paper is to use Ricci flow to study convex surfaces in the sense of Alexandrov. This kind of surfaces is a generalization of smooth convex surfaces to non-smooth case, and share many similarities with the latter. One example is Pogorelov's famous rigidity theorem \cite{Pogo rigit}, which generalizes Cohn-Vesson's classical rigidity result \cite{CohnV} about smooth convex surfaces to Alexandrov sense. Since there is no regularity assumption on such surfaces, Pogorelov's theorem is not only non-trivial but also difficult to access. In the future work, we hope to use the results in this paper to study Pogorelov's rigidity theorem.

The paper is organized as follows. In section \ref{Preliminaries}, we recall some basic facts about convex surfaces in the sense of Alexandrov. Then, in the third section, we prove the existence and uniqueness of Ricci flow in the sense of Theorem \ref{Existence}. Finally, we introduce our project which aims to study the rigidity of convex surfaces by using Ricci flow.

\medskip

{\bf Acknowledgements.} The authors would like to convey their gratitude to their supervisor Professor Pengfei Guan for suggesting this problem and his attentive guidance and several valuable comments on improving this paper. 
\section{Preliminaries}\label{Preliminaries}

In this section, we recall some basic results about convex surfaces in the sense of Alexandrov. These are mainly taken from \cite{Alexandrov II, Pogo extr}, see also the Appendix of \cite{Richard2}. 

Let $(X,d)$ be a metric space, it is called a  geodesic metric space if any two points $a$ and $b$ in $X$ can be connected by a continuous path of shortest length on $X$. Suppose $(X_1,d_1)$ and $(X_2,d_2)$ are two metric spaces, an isometry between $X_1$ and $X_2$ is a bijection $f: X_1\to X_2$ such that 
\begin{equation*}
d_2(f(a),f(b))=d_1(a,b),\ \ \ for\ all\ a,\ b\in X_1.
\end{equation*}

Let $a$, $b$ and $c$ be three different points in a geodesic metric space $(X,d)$, we define the comparison angle $\tilde{\angle} a_b^c$ as the angle at $\tilde{a}$ of the comparison triangle $\tilde{a}\tilde{b}\tilde{c}$ in $S_0$  whose sides have length $d_0(\tilde{a},\tilde{b})=d(a,b)$, $d_0(\tilde{a},\tilde{c})=d(a,c)$ and $d_0(\tilde{b},\tilde{c})=d(b,c)$, where $S_0$ is the Euclidean space, and $d_0$ is the standard distance in $S_0$.

\begin{defi}
Let $(X,d)$ be a geodesic metric space, it is said to satisfy the convexity condition if for any point $a\in X$,  and any two shortest paths $(\gamma_1(s))_{s\in[0,T]}$ and $(\gamma_2(s))_{s\in[0,T]}$ in $X$ parametrized by arc length issuing from $a$,  the comparison angle $\tilde{\angle}a_{\gamma_1(s)}^{\gamma_2(t)}$ is an non-increasing function of $s$ and $t$.
\end{defi}
\begin{defi}
Let $(X,d)$ be a geodesic metric space, it is called a closed convex surface in the sense of Alexandrov, if it is at the same time a compact topological surface without boundary, and satisfies the convexity condition.
\end{defi}

We also have the following equivalent definition.

\begin{defi}
A closed convex surface in the sense of Alexandrov is a geodesic metric space $(X,d)$ which is at the same time a compact topological surface without boundary and a metric space with non-negative curvature in the sense of Alexandrov.\end{defi}
A geodesic metric space has non-negative curvature in the sense of Alexandrov if its geodesic triangles are bigger than the geodesic triangles in $S_0$. To be more precise, a geodesic metric space $(X, d)$ has non-negative curvature in the sense of Alexandrov if and only if the following condition is satisfied:

{\it Let $a$, $b$ and $c$ be any three points in $(X,d)$, and $m$ be any point on a shortest path from $b$ to $c$. Let $\tilde{a}$, $\tilde{b}$ and $\tilde{c}$ be points in $S_0$ such that $d_0(\tilde{a},\tilde{b})=d(a,b)$, $d_0(\tilde{a},\tilde{c})=d(a,c)$ and $d_0(\tilde{b},\tilde{c})=d(b,c)$. If $\tilde{m}$ is a point on  $\tilde{b}\tilde{c}$ such that $d_0(\tilde{b},\tilde{m})=d(b,m)$. Then $d(a,m)\geqslant d_0(\tilde{a},\tilde{m})$.}

In the following, we call a closed convex surface in the sense of Alexandrov a closed convex surface if there is no confusion.
By Toponogov's theorem, every closed smooth surface  with non-negative Gauss curvature is a closed convex surface. The boundary of a convex set with the induced metric  in $\mathbb{R}^3$ is also a closed convex surface (Theorem $10.2.6$ in \cite{BBI}). Alexandrov proved the following isometric embeding theorem.
\begin{thm}\label{Alex embedding}(Page 269 in \cite{Alexandrov II}) Any closed convex surface $(X,d)$ can be isometrically embedded into $\mathbb{R}^3$ as the boundary of a (possibly degenerate) convex body. 
\end{thm}

The following lemma will be used in the proof of Theorem \ref{Existence}. Its geometric meaning is that the Hausdorff distance of two convex surfaces in $\mathbb{R}^3$ controls their intrinsic distance functions.
\begin{lem}\label{Hauss cont intrin} (Theorem $2$ in Chapter $3$ of \cite{Alexandrov II}) 
 For every closed convex surface $F$ and for every $\varepsilon>0$, there exists a $\delta>0$ such that whenever the deviation of a closed convex surface $S$ from $F$ is less than $\delta$ and the distance of some points $X$ and $Y$ on $F$ from some points $A$ and $B$ on $S$ are also less than $\delta$, we have 
\begin{equation*}
  |d_{F}(X,Y)-d_{S}(A,B)|<\varepsilon,
\end{equation*}
 where $d_{F}$ and $d_{S}$ are the distance functions on $F$ and $S$ respectively.
\end{lem}

\section{Existence and Uniqueness of Ricci flow}\label{ex and un}
In this section, assuming that $(X,d)$ is an embedded closed convex surface in $\mathbb{R}^3$, we prove the existence and uniqueness of the Ricci flow that admits $(X,d)$ as a metric initial condition.

\subsection{Existence of the Ricci flow}\label{sec existen}
In this subsection, we prove the existence of Ricci flow $(X,g(t))_{t\in(0,T)}$ that admits $(X,d)$ as metric initial condition.  First, note that, for every embedded closed convex surface $(X,d)$ in $\mathbb{R}^3$, it can be seen as the boundary of a convex body $K$. Then there is a sequence of smooth convex bodies $K_i$ that converges to $K$ in Hausdorff distance (Theorem $3.4.1$ in \cite{Schneider}). In particular, $\partial K_i$ converges to $X$ in Hausdorff distance as well. By Lemma \ref{Hauss cont intrin}, $\partial K_i$ with induced metric in $\mathbb{R}^3$ (denoted by $(X_i,\tilde{g}_i)$) converge to $(X,d)$ in intrinsic sense. More precisely, we have the following lemma.
\begin{lem} \label{approximation C0}
 Let $(X,d)$ be an embedded closed convex surface in $\mathbb{R}^3$. Then there exists a sequence of closed smooth convex surfaces $\{(X_i,\tilde{g}_i)\}_{i=1}^{\infty}$ and bijections $f_i:X\rightarrow X_i$ such that $d_i(x, y)$ converges to $d(x,y)$ uniformly for $x,\ y\in X$ as $i\to\infty$, where $d_i(x,y)=\tilde{d}_i(f_i(x),f_i(y))$, and $\tilde{d}_i$ is the distance on $X_i$ induced by $\tilde{g}_i$.
\end{lem}
Now we prove  Theorem \ref{Existence} by using Lemma  \ref{approximation C0}.

\medskip
  
{\it Proof of Theorem \ref{Existence}.}\ \  First, for a sequence of subsets in the same metric space, Gromov-Hausdorff distance by definition is not greater than Hausdorff distance. Thus the Hausdorff convergence of $(X_i,\tilde{g}_i)$ to $(X,d)$ in Lemma $\ref{approximation C0}$ implies the Gromov-Hausdorff convergence of $(X_i,\tilde{g}_i)$ to $(X,d)$. Suppose 
\begin{equation}
diam(X,d)\leqslant D\ \ \ and\ \ \ vol(X)\geqslant\tilde{\nu}_0>0
 \end{equation} 
for some positive constants $D$ and $\tilde{\nu}_0$. Since the diameter and volume are continuous with respect to Gromov-Hausdorff convergence with sectional curvature bounded from below, and all the surfaces $(X_i,\tilde{g}_i)$ and $(X,d)$ are convex, we have  
\begin{equation}
diam(X_i,\tilde{g}_i)\leqslant 2D\ \ \ and\ \ \ ^{\tilde{g}_i}vol(X_i)\geqslant\frac{\tilde{\nu}_0}{2}>0
 \end{equation} 
for $(X_i,\tilde{g}_i)$ with large $i$. By Bishop-Gromov comparison theorem, we get
 \begin{equation}
\frac{vol(^{\tilde{g}_i}B_1(x))}{vol(B_1(x))}\geqslant\frac{vol(^{\tilde{g}_i}B_{2D}(x))}{vol(B_{2D}(x))}=\frac{^{\tilde{g}_i}vol(X_i)}{vol(B_{2D}(x))}\geqslant\frac{\tilde{\nu}_0}{2vol(B_{2D}(x))},
 \end{equation} 
  which implies that
   \begin{equation}
vol(^{\tilde{g}_i}B_1(x))\geqslant\tilde{\nu}_0\frac{vol(B_1(x))}{2vol(B_{2D}(x))}:=\nu_0>0.
 \end{equation} 
Due to Theorem \ref{simon}, we know that there are smooth Ricci flows $(X_i,g_i(t))_{t\in[0,T)}$ with $g_i(0)=\tilde{g}_i$ satisfying
\begin{equation*}
\begin{split}
&(a'_t)\ Ricci (g_i(t))\geqslant 0\ for\ all\ t\in [0,T);\\
&(b'_t)\ vol(^{g_i(t)}B_1(x))\geqslant \frac{\nu_0}{2}>0,\ for\ all\ x\in X\ and\ t\in [0,T);\\
&(c'_t)\ \sup_{X_i}|Riem(g_i(t))|\leqslant \frac{K}{t}\ for\ all\ t\in[0,T);\\
&(d'_t)\ d_{g_i(s)}(p,q)-c_2(\sqrt{t}-\sqrt{s})\leqslant d_{g_i(t)}(p,q)\leqslant e^{c_1(t-s)}d_{g_i(s)}(p,q),\\
&\ \ \ \ \ \ for\ all\ 0\leqslant s\leqslant t\in[0,T)\ and\ p,\ q\in X_i,
\end{split}
\end{equation*} 
where $K=K(\nu_0)$, $c_1=c_1(\nu_0)$, $c_2=c_2(\nu_0)$ and $T=T(\nu_0)$ are constants independent of $i$. Combining $(c'_t)$ and Shi's higher derivative estimates \cite{Shi1, Shi2} with Arzela-Ascoli theorem, there exists a subsequence (which we also denote by $g_i(t)$) converges to a metric $g(t)$ on $X$, and $(X,g(t))_{t\in(0,T)}$ is a smooth Ricci flow. Let $s\to0$ in $(d'_t)$, for all $ t\in(0,T)$ and $p,\ q\in X$, we have
\begin{equation}
\tilde{d}_{i}(f_i(p),f_i(q))-c_2\sqrt{t}\leqslant d_{g_i(t)}(f_i(p),f_i(q))\leqslant e^{c_1t}\tilde{d}_{i}(f_i(p),f_i(q)).
\end{equation} 
Since $X_i$ converges to $X$ in Hausdorff distance and $g_i(t)$ converges to $g(t)$ in local smooth sense of $(0,T)$, by letting $i\to\infty$ and using Lemma \ref{approximation C0}, we have
\begin{equation}\label{041601}
d(p,q)-c_2\sqrt{t}\leqslant d_{g(t)}(p,q)\leqslant e^{c_1t}d(p,q).
\end{equation} 
Then the uniform convergence of $d_{g(t)}$ to $d$ follows by letting $t\to0$ in $(\ref{041601})$.\QEDB

\subsection{Uniqueness}
  
In this subsection, we prove the uniqueness of Ricci flow that admits a non-degenerate embedded closed convex surface in $\mathbb{R}^3$ as a metric initial condition.  The point is to give an exact expression of the initial metric $d$ (Theorem \ref{01012}) firstly, and then to prove that the isometry between the two metric initial conditions is differentiable (Theorem \ref{01010}).

Let $(X,g(t))_{t\in(0,T)}$ be the Ricci flow obtained in subsection \ref{sec existen}. It is easy to see that its Gaussian curvature $K_{g(t)}$ is positive. In fact, the non-negativity of the Gaussian curvature $K_{\tilde{g}_i}$ of the smooth convex surfaces $(X_i,\tilde{g}_i)$ implies that the Gauss curvature $K_{g_i(t)}$ along Ricci flow $(X_i, g_i(t))_{t\in[0,T)}$ is also non-negative by applying maximum principle to the evolution equation of $K_{g_i(t)}$,
\begin{equation}
\frac{\partial}{\partial t}K_{g_i(t)}=\Delta_{g_i(t)}K_{g_i(t)}+|Ric_{g_i(t)}|^2_{g_i(t)}.
\end{equation} 
From Gauss-Bonnet theorem for $g_i(t)$ and the fact that $g_i(t)$ converges to $g(t)$ smoothly, for $t\in(0,T)$, $K_{g(t)}$ is non-negative and satifies
\begin{equation}
\int_{X}K_{g(t)}dV_{g(t)}=4\pi.
\end{equation}
Hence there must be a point $x_0$ such that $K_{g(t)}$ is positive at $x_0$. Then strong maximum principle implies that $K_{g(t)}$ is positive everywhere.

Fix $t_0\in(0,T)$ and denote $(X,g(t_0))=(X,g_{t_0})$. By the Uniformization theorem, there is a conformal equivalence (holomorphic isomorphism) $\Phi: (\mathbb{S}^2,\tilde{h})\rightarrow (X,g_{t_0})$, where  $\tilde{h}$ is a smooth metric of positive constant curvature, i.e. $\Phi^*(g_{t_0})=e^{\tilde{u}(t_0,x)}\tilde{h}(x)$ for some smooth function $\tilde{u}(t_0,x)$ on $\mathbb{S}^2$. On the other hand, the 2-dimensional Ricci flow can be written as 
\begin{equation}\label{evo g 2}
\frac{\partial}{\partial t}g(t)=-R_{g(t)}g(t).
\end{equation} 
Hence we can write $g(t)=e^{-\int_{t_0}^tR_{g(s)}ds}g_{t_0}:=\tilde{\omega}(t,x)g_{t_0}$, which implies that
\begin{equation*}
g(t)=\tilde{\omega}(t,x)(\Phi^{-1})^*(e^{\tilde{u}(t_0,x)}\tilde{h}(x))=\tilde{\omega}(t,x)e^{\tilde{u}(t_0,\Phi^{-1}(x))}\tilde{h}(\Phi^{-1}(x)):=e^{2u(t,x)}h(x).
\end{equation*}  
where $h(x):=\tilde{h}(\Phi^{-1}(x))$ and $u(t,x):=\frac{\ln(\tilde{\omega}(t,x))+\tilde{u}(t_0,\Phi^{-1}(x))}{2}$. We call $u(t):=u(t,x)$ the conformal potential along Ricci flow $(X,g(t))_{t\in(0,T)}$.  From (\ref{evo g 2}), the evolution equation of $u(t)$ reads
\begin{equation}
\frac{\partial u(t)}{\partial t}=e^{-2u(t)}(\Delta_hu(t)-K_h)=-K_{g(t)}.
\end{equation}
Since $K_{g(t)}$ is positive, $u(t)$ increases as $t$ decreases to $0$, and $u(t)\geqslant u(T)$ for $t\in(0,T]$.  It is proved in Lemma $2.2$ of \cite{Richard2} that $u(t)$ is uniformly bounded in $L^1$-sense and converges to an integrable function $u_0(x)$ in $L^1$-sense. Here, we prove that $u(t)$ is uniformly bounded for $t\in(0,T]$ in the classical sense (so $u_0(x)$ is also bounded in the classical sense).  

Before starting the proof, we remark that every Ricci flow $(X,g(t))_{t\in(0,T)}$ admitting $(X,d)$ as metric initial condition can be obtained through the process in subsection \ref{sec existen}. This is due to Proposition $0.6$ in \cite{Richard2}. Hence we only need to consider the uniqueness for the Ricci flow obtained in subsection \ref{sec existen}.
\begin{lem}\label{init u bounded}
Assume that $(X,d)$ is a non-degenerate embedded closed convex surface in $\mathbb{R}^3$. Let $(X,g(t))_{t\in(0,T]}$ be a Ricci flow admitting $(X,d)$ as metric initial condition. Then the conformal potential $u(t)$ along $(X,g(t))_{t\in(0,T]}$ is uniformly bounded. Thus, $u_0(x)$ is bounded on $X$.
\end{lem}
 
To prove this lemma, we first prove that the smooth approximating metrics $\tilde{g}_i$ obtained in Lemma \ref{approximation C0} are uniformly equivalent to the standard metric $\delta$ on the unit sphere $\mathbb{S}^2$. Let $(X,g)$ be a smooth embedded closed convex surface in $\mathbb{R}^3$, 
$\rho$ be the radial function of $(X,g)$, and $v=\frac{1}{\rho}$. Then the induced metric on $X$ and the second fundamental form can be written as
\begin{equation}\label{1 2 fundmen}
\begin{split}
&g_{ij}=\rho^2\delta_{ij}+\rho_i\rho_j,\\
& h_{ij}=\frac{1}{\sqrt{\rho^2+|\nabla_{\delta}\rho|^2}} (\rho^2\delta_{ij}+2\rho_i\rho_j-\rho\rho_{ij})\\
& \ \ \ \ =\frac{\rho^3}{\sqrt{\rho^2+|\nabla_{\delta}\rho|^2}} (v_{ij}+v\delta_{ij}),
\end{split}
\end{equation}
where the derivatives are taken with respect to the connection of
 $(\mathbb{S}^2,\delta)$.
\begin{lem}\label{01001}
There is a uniform constant $C$ such that 
\begin{equation}
\frac{1}{C}\delta\leqslant\tilde{g}_i\leqslant C\delta\ \ \ for\ large\ i.
\end{equation} 
\end{lem}
\begin{proof}
 Let $^i\rho$ be the radial function of $(X_i,\tilde{g}_i)$ and $^iv=\frac{1}{^i\rho}$. We claim that 
\begin{equation}
\max_{\mathbb{S}^2}(|\nabla_\delta\ ^iv|^2+\ ^iv^2)\leqslant \max_{\mathbb{S}^2}\ ^iv^2.
\end{equation}  
Define $f=|\nabla_\delta\ ^iv|^2+(1+\eta)\ ^iv^2$ for $\eta>0$. At the maximum point of $f$, we have
\begin{equation}
0=\nabla_{\delta l} f=\nabla_{\delta l}(|\nabla_\delta\ ^iv|^2+(1+\eta)\ ^iv^2)=2\ ^iv_j(\ ^iv_{lj}+(1+\eta)\ ^iv\delta_{lj}).
\end{equation}  
Since $(X_i,\tilde{g}_i)$ is convex, $(\ ^ih_{lj})\geqslant0$ and then $(^iv_{lj}+\ ^iv\delta_{lj})\geqslant0$. Then we have $(^iv_{lj}+(1+\eta)\ ^iv\delta_{lj})>0$. Hence $\nabla_\delta v=0$ at the maximum point of $f$, which implies 
\begin{equation}
\max_{\mathbb{S}^2}(|\nabla_\delta\ ^iv|^2+(1+\eta)\ ^iv^2)\leqslant (1+\eta)\max_{\mathbb{S}^2}\ ^iv^2.
\end{equation} 
Let $\eta\to0$, we complete the proof of the claim.

Since $(X,d)$ is non-degenerate and $(X_i,\tilde{g}_i)$ converges to $(X,d)$ in Hausdorff sense, there exists a constant $C$ such that for large $i$, 
\begin{equation}
\frac{1}{C}\leqslant\ ^i\rho\leqslant C.
\end{equation}  
From the claim, $|\nabla_\delta\ ^i\rho|^2$ are also uniformly bounded for large $i$. Taking trace with respect to $\delta$ on both sides of $(\tilde{g}_i)_{lj}=\ ^i\rho^2\delta_{lj}+\ ^i\rho_l\ ^i\rho_j$, we conclude that there exists uniform constant $C$ such that
\begin{equation}
tr_\delta \tilde{g}_i\leqslant C\ \ \ and\ \ \ tr_\delta \tilde{g}_i\geqslant\frac{1}{C}>0\ \ for\ large\ i.
\end{equation} 
Therefore, there is a uniform constant $C$ such that 
\begin{equation}
\frac{1}{C}\delta\leqslant\tilde{g}_i\leqslant C\delta
\end{equation} 
for large $i$.
\end{proof}

\medskip

Next, we extend this equivalence to Ricci flow.
\begin{lem}\label{01002}
Let $(X,g(t))_{t\in[0,T]}$ be a 2-dimensional smooth Ricci flow with initial metric $g_0$, then we have
\begin{equation}\label{01003}
g(t)\leqslant e^{L_1T}g_0\ \ \ and\ \ \ tr_{g(t)}g_0\leqslant L_2(\frac{dV_{g_0}}{dV_{g(T)}})^2,
\end{equation}  
where $-L_1$ is the lower bound of $R_{g_0}$, and $L_2$ depends on $L_1$, $T$.
\end{lem}
\begin{proof}
Applying maximum principle to the evolution equation of the scalar curvature $R_{g(t)}$
\begin{equation}
\frac{\partial}{\partial t}R_{g(t)}=\Delta_{g(t)}R_{g(t)}+2|Ric_{g(t)}|_{g(t)}^2,
\end{equation}  
we get
\begin{equation}
R_{g(t)}\geqslant R_{g_0}\geqslant-L_1.
\end{equation}
In dimension $2$, Ricci flow can be written as
\begin{equation}
\frac{\partial}{\partial t}g(t)=-R_{g(t)}g(t)\leqslant L_1g(t).
\end{equation} 
Hence we have $g(t)\leqslant e^{L_1t}g_0$ and get the first estimate in $(\ref{01003})$.

For the second estimates, we need the following inequality.
\begin{equation}
n\Big(\frac{det g_1}{det g_2}\Big)^{\frac{1}{n}}\leqslant tr_{g_2}g_1\leqslant n\Big( \frac{det g_1}{det g_2}\Big)(tr_{g_1}g_2)^{n-1},
\end{equation}
where $g_1$ and $g_2$ are any two smooth $n$-dimensional metrics. In our case,
\begin{equation}
tr_{g(t)}g_0\leqslant 2\Big( \frac{det g_0}{det g(t)}\Big)(tr_{g_0}g(t))=2\Big( \frac{dV_{g_0}}{dV_{g(t)}}\Big)^2(tr_{g_0}g(t)).
\end{equation}
Hence we only need to prove that $dV_{g(t)}$ is bounded from below uniformly. The volume form evolves as
\begin{equation}
\frac{\partial}{\partial t}dV_{g(t)}=-RdV_{g(t)}\leqslant L_1dV_{g(t)},
\end{equation}  
which implies that $e^{-L_1t}dV_{g(t)}$ decrease  and then $dV_{g(t)}\geqslant e^{-L_1(T-t)}dV_{g(T)}\geqslant e^{-L_1T}dV_{g(T)}$. So we have
\begin{equation}
tr_{g(t)}g_0\leqslant 4\Big( \frac{dV_{g_0}}{e^{-L_1T}dV_{g(T)}}\Big)^2e^{L_1T}.
\end{equation}  
Let $L_2=4e^{3L_1T}$, we complete this Lemma.
\end{proof}

\medskip

For the sequence of smooth Ricci flow $(X_i, g_i(t))_{t\in[0,T]}$ with $\tilde{g}_i$ as initial condition, we have by Lemma \ref{01001} and Lemma \ref{01002}, that
\begin{equation}\label{01004}
g_i(t)\leqslant C\delta\ \ \ and\ \ \ tr_{g_i(t)}\delta\leqslant 4C^3(\frac{dV_\delta}{dV_{g_i(T)}})^2,
\end{equation} 
where $C$ is the constant in Lemma \ref{01001}.  Letting $i\to\infty$ gives 
\begin{equation}\label{01005}
g(t)\leqslant C\delta\ \ \ and\ \ \ tr_{g(t)}\delta\leqslant 4C^3(\frac{dV_\delta}{dV_{g(T)}})^2 \ \ \ for\ t\in(0,T],
\end{equation} 
which is equivalent to
\begin{equation}\label{01006}
\frac{1}{A}\delta\leqslant g(t)\leqslant C\delta\ \ \ for\ t\in(0,T],
\end{equation} 
where $A$ depends on $C$ and $T$. In fact, we proved the following Lemma.
\begin{lem}\label{01007}
Assume that $(X,d)$ is a non-degenerate embedded closed convex surface in $\mathbb{R}^3$. Let $(X,g(t))_{t\in(0,T]}$ be a Ricci flow admitting $(X,d)$ as a metric initial condition. Then there exists a constant $C$ such that 
\begin{equation}\label{01008}
\frac{1}{C}\delta\leqslant g(t)\leqslant C\delta\ \ \ for\ t\in(0,T].
\end{equation}
\end{lem}
Now Lemma \ref{init u bounded} follows immediately.

\medskip

{\it Proof of Lemma \ref{init u bounded}.}\ \ Since $g(t)=e^{2u(t)}h(x)$ is uniform equivalent to $\delta$ for $t\in(0,T]$, $u(t)$ is uniformly bounded. Since $u(t)$ increases to $u_0$ as $t$ decreases to $0$, $u_0$ is also bounded. \QEDB

\medskip

Next, for an $L^1$-function $u$ and a smooth Riemannian metric $h$ on $M$, there is a metric $d_{h,u}$ defined as
\begin{equation}\label{01009}
d_{h,u}(x,y)=\inf_{\gamma\in\Gamma(x,y)}\int_0^1e^{u(\gamma(\tau))}|\dot{\gamma}(\tau)|_hd\tau,
\end{equation}  
where $\Gamma(x,y)$ is the space of $C^1$ paths $\gamma$ from $[0,1]$ to $M$ with $\gamma(0)=x$ and $\gamma(1)=y$. This metric was studied by Reshetnyak \cite{RESHE}. For more details, please see the appendix of \cite{Richard2}.

In \cite{Richard2}, when $(X,d)$ is a compact Alexandrov surface with curvature bounded from below, the metric $d_{g(t)}$ induced by $g(t)$ (here, $g(t)$ is a Ricci flow on $M$ with metric initial condition $(X,d)$ in the sense of Theorem \ref{richard}) converges to $d_{h,u_0}$ uniformly, where $u_0$ is the $L^1$-limit of the conformal potential $u(t)$ along Ricci flow $(M,g(t))_{t\in(0,T]}$ as $t\to0$ in \cite{Richard2}. But $d$ may not be $d_{h,u_0}$ there. In fact, we  can only conclude that $(M,d_{h,u_0})$ is isometric to $(X,d)$ from the Lemma $2.4$ in \cite{Richard2} . In our case when $(X,d)$ is a non-degenerate embedded closed convex surface in $\mathbb{R}^3$, we prove that indeed $d=d_{h,u_0}$.

\begin{thm}\label{01012}
Assume that $(X,d)$ is a non-degenerate embedded closed convex surface in $\mathbb{R}^3$. Let $(X,g(t))_{t\in(0,T]}$ be the Ricci flow admitting $(X,d)$ as a metric initial condition and $u(t)$ be the conformal potential along $(X,g(t))_{t\in(0,T]}$. Then 
\begin{equation}\label{0415}
d=d_{h,u_0},
\end{equation}
where $u_0(x)$ is the pointwise limit of $u(t,x)$ as $t\to0$. 
\end{thm}
\begin{proof}
By definition $(\ref{01009})$,  and the definition of confromal potential $g(t)=e^{2u(t,x)}h(x)$, we have
\begin{equation}\label{01011}
\begin{split}
d_{h,u(t)}(x,y)&=\inf_{\gamma\in\Gamma(x,y)}\int_0^1e^{u(t,\gamma(\tau))}|\dot{\gamma}(\tau)|_hd\tau\\
&=\inf_{\gamma\in\Gamma(x,y)}\int_0^1|\dot{\gamma}(\tau)|_{g_t(\tau)}d\tau=d_{g(t)}(x,y).
\end{split}
\end{equation}  
By Lemma $2.4$ in \cite{Richard2}, $d_{h,u(t)}$ converges to $d_{h,u_0}$ uniformly as $t\to0$. Since we have proved that $d_{g(t)}$ converges to $d$ uniformly as $t\to0$ on $X$, then $(\ref{0415})$ follows by letting $t\to0$ on both sides of $(\ref{01011})$.
\end{proof}

\medskip

We now prove a regularity theorem for the isometry between the metric spaces $(X_1, e^{2u_1}h_1)$ and $(X_2, e^{2u_2}h_2)$, where $u_1$ and $u_2$ are two bounded functions, and $h_1$ and $h_2$ are pull back metrics of  two metrics on $\mathbb{S}^2$ with constant Gaussian curvature.
\begin{thm}\label{01010}
Assume that $F:(X_1, e^{2u_1}h_1)\rightarrow(X_2, e^{2u_2}h_2)$ is an isometry.  Then $F$ is differentiable, where $u_1$ and $u_2$ are two bounded functions, and $h_1$ and $h_2$ are two pull back metrics of the metrics on $\mathbb{S}^2$ with constant Gaussian curvature.
\end{thm}
\begin{proof}
Since $F$ is an isometry, $F$ is bi-Lipschitz and then it is differentiable almost everywhere. Write locally $h_1=\lambda_1(du^2+dv^2)$ and $h_2=\lambda_2(dx^2+dy^2)$ for some positive functions $\lambda_1$ and $\lambda_2$. At differentiable point of $F$, we have
\begin{equation*}
\begin{split}
e^{2u_1}\lambda_1(du^2+dv^2)&=e^{2u_1}h_1=F^*(e^{2u_2}h_2)=e^{2u_2\circ F}(\lambda_2\circ F) F^*h_2\\
&= e^{2u_2\circ F}(\lambda_2\circ F)\big((x_udu+x_vdv)^2+(y_udu+y_vdv)^2\big)\\
&= e^{2u_2\circ F}(\lambda_2\circ F)\big((x^2_u+y^2_u)du^2+(x^2_v+y^2_v)dv^2+2(x_ux_v+y_uy_v)dudv\big)
\end{split}
\end{equation*}  
Since $u_1$ and $u_2$ are bounded functions, we have
\begin{equation}
x_ux_v+y_uy_v=0\ \ \ and\ \ \ x_u^2+y_u^2=x_v^2+y_v^2,
\end{equation}  
which is equivalent to the fact that 
\begin{equation}
x_u=-y_v\ and\ x_v=y_u\ \ \ or\ \ \ x_u=y_v\ and\ x_v=-y_u.
\end{equation} 
Then we know that $F$ is differentiable everywhere  by using the analytic extension theorem in \cite{Arsove}.
\end{proof}

\medskip

Since $u(t)$ increases to $u_0$ and both of them are bounded, by using $(\ref{01011})$ and Lebesgue's monotone convergence theorem, we conclude that $d=d_{h,u_0}$ is the metric induced by $e^{2u_0}h$.

\medskip   

{\it Proof of Theorem \ref{uni1}.} \ \ By Theorem \ref{01012}, $d_i=d_{h_i,u_i}$ $(i=1,2)$ is the metric induced by $e^{2u_i}h_i$, where $u_i=\lim\limits_{t\to 0}u_i(t)$, $u_i(t)$ is the conformal potential along $(X_i,g_i(t))_{t\in(0,T]}$ and $h_i$ is the pull back metric of a metric on $\mathbb{S}^2$ with constant Gaussian curvature. 

From Theorem \ref{01010}, the isometry $f$ is differentiable. Then $(X_1,f^*g_2(t))_{t\in(0,T)}$ is also a Ricci flow admitting $(X_1,d_1)$ as metric initial condition in the sense that the  distance fucntion induced by $f^*g_2(t)$ converges uniformly to $d_1$ as $t\to0$. By using Proposition $0.6$ in \cite{Richard2}, we have $g_1(t)=f^*g_2(t)$.\QEDB

\section{Further discussions}\label{further discussions}

In this section, we introduce our project which aims to study Pogorelov's uniqueness theorem by Ricci flow. 

Pogorelov's uniqueness theorem  \cite{Pogo extr} states that any two closed isometric convex surfaces with induced metrics in $\mathbb{R}^3$ are congruent. It is a generalization of the classical Cohn-Vesson's rigidity theorem \cite{CohnV}.

Our basic idea to study this theorem is to use Ricci flow to construct two families of smooth convex surfaces approximating the two isometric closed convex surfaces in Pogorelov's theorem. We then apply Cohn-Vesson's rigidity theorem to the smooth surfaces and take a limit as $t\to0$ to see if ``the limit of Cohn-Vesson's rigidity theorem will imply Pogorelov's rigid theorem". More precisely, given two isometric embedded closed convex surfaces $(X_1, d_1)$ and $(X_2, d_2)$ in $\mathbb{R}^3$, Theorem \ref{Existence} implies that there are two Ricci flows $(X_1,g_1(t))_{t\in(0,T)}$ and $(X_2,g_2(t))_{t\in(0,T)}$ admitting $(X_1, d_1)$ and $(X_2, d_2)$ as metric initial conditions. Then for every positive time $t$, we embed the Ricci flow $(X_1,g_1(t))$ and $(X_2,g_2(t))$ smoothly and isometrically into $\mathbb{R}^3$ as $(X^t_1,G_1(t))$ and $(X^t_2,G_2(t))$ respectively. The validity for these embeddings is due to the fact that $(X_1,g_1(t))$ and $(X_2,g_2(t))$ are smooth strictly convex surfaces and the solvability of Weyl's problem proved by Nirenberg \cite{Nirenberg}. By Theorem \ref{uni1}, $(X^t_1,G_1(t))$ and $(X^t_2,G_2(t))$ are isometric. Then Cohn-Vesson's rigidity theorem implies that there is a congruence $F(t)\in O(3)$ between them.  We hope to investigate the limits of $(X^t_1,G_1(t))$ and $(X^t_2,G_2(t))$ as $t\to0$. If $(X^t_1,G_1(t))$ and $(X^t_2,G_2(t))$ converge to $(X_1, d_1)$ and $(X_2, d_2)$ in Hausdorff distance up to an isometry in $O(3)$ respectively, the compactness of $O(3)$ will imply the congruence between $(X_1, d_1)$ and $(X_2, d_2)$.

\end{document}